\documentclass[12pt,a4paper]{article}
\usepackage{amsmath,amssymb,amsthm}
\usepackage{amssymb,latexsym, bbm,comment}
\usepackage{graphicx}
\usepackage{makeidx}
\newtheorem{theorem}{Theorem}[section]
\newtheorem{prop}[theorem]{Proposition}
\numberwithin{equation}{section}

\newcommand{\R}{\mathbb{R}}
\newcommand{\Rd}{\mathbb{R}^d}

\newcommand{\E}{\mathbb{E}}

\newcommand{\cadlag}{c\`adl\`ag}

\newcommand{\g}{\mathfrak{g}}

\newcommand{\p}{\mathfrak{p}}
\newcommand{\fk}{\mathfrak{k}}
\newcommand{\tr}{\mbox{tr}}

\newcommand{\bean}{\begin{eqnarray*}}
\newcommand{\eean}{\end{eqnarray*}}

\newcommand{\G}{\widehat{G}}

\newcommand{\Ad}{\rm{Ad}}

\begin{document}

\date{}

\title{Markov Processes with Jumps on Manifolds and Lie Groups}

\author{David Applebaum\footnote{D.Applebaum@sheffield.ac.uk}\\
 School of Mathematics and Statistics,\\ University of Sheffield, \\
Sheffield S3 7RH\\
United Kingdom.\\
 \and
    Ming Liao\footnote{liaomin@auburn.edu}\\
Department of Mathematics,\\ Auburn University, Auburn, AL, USA}

\maketitle

\begin{abstract}
We review some developments concerning Markov and Feller processes with jumps in geometric settings. These include stochastic differential equations in Markus canonical form, the Courr\`{e}ge theorem on Lie groups, and invariant Markov processes on manifolds under both transitive and more general Lie group actions.
\end{abstract}

\begin{center}
{\it We dedicate this article to the memory of Hiroshi Kunita (1937--2019).}
\end{center}

\section{Introduction}

Stochastic differential geometry is a deep and beautiful subject. It is essentially the study of stochastic processes that take their values in manifolds, and so it naturally sits at the intersection of probability theory with differential geometry, but it also impacts on, and makes use of techniques from real, stochastic and functional analysis, dynamical systems and ergodic theory. If the manifold has the additional structure of being a Lie group, then more tools are available and the results are of considerable interest in their own right.  By far the majority of work on the subject has arisen from studying Markov processes that arise as the solutions of stochastic differential equations (SDEs) on manifolds or Lie groups. From its beginnings with the pioneering work of It\^{o} in 1950 \cite{It} until the 1990's, the emphasis was on processes with continuous sample paths that are obtained by solving SDEs driven by Brownian motion. For accounts of this work, see \cite{El0, El, Hsu} and references therein.  More recently, there has been increasing interest in studying processes with jumps which are solutions of SDEs driven by L\'{e}vy processes. The material reviewed in this paper is wholly concerned with the jump case, including L\'{e}vy processes in Lie groups and manifolds, and generalisations.

The organisation of the paper is as follows. we begin with a short section 2 that reviews the key definitions of Markov and Feller processes in a suitably general setting. In section 3, we describe SDEs driven by L\'{e}vy processes on manifolds in Markus canonical form. If unique solutions exist, then they give rise to a Markov process. As an example we show how to obtain a L\'{e}vy process on a compact manifold by projection from the solution of an SDE on the frame bundle. When there are no jumps, this is precisely the celebrated Eels--Elworthy construction of Brownian motion on a manifold. The next section is more analytic. We outline the proof of a global version of the Courr\`{e}ge theorem in a Lie group, which gives a canonical form for a linear operator that satisfies the positive maximum principle (PMP). The probabilistic importance of this result is that the generators of all sufficiently rich Feller processes satisfy the PMP, and so their generators must be of this form. Indeed we see that the generators are characterised by a real--valued function, a vector--valued function, a matrix--valued function, and a kernel, that may be probabilistically interpreted as describing killing, drift, diffusion and jump intensity (respectively). We also describe how, when the group is compact, the generator may be represented by a pseudo--differential operator in the Ruzhansky--Turunen sense.

Sections 5 and 7 deal with Markov processes that are suitably invariant (i.e. their transition probabilities are invariant) under the action of a Lie group. In section 5, we examine the case where the group acts transitively. In this case the manifold is a homogeneous space and the Markov process is, in fact, a L\'{e}vy process. To consider the non--transitive case, we need the notion of inhomogeneous L\'{e}vy process in a homogeneous space, i.e. a process that has independent, but not necessarily stationary increments. These are described more fully in the short section 6. In section 7, we consider the non--transitive case where we may effectively assume that the manifold $M$ is the product of another manifold $M_{1}$ and a homogeneous space $M_{2}$. Then our process is the product of a radial part, that lives in $M_{1}$, and an angular part that lives in $M_{2}$. In fact the radial part is a Markov process, and the angular part is an inhomogeneous L\'{e}vy process. In the case of sample path continuity, some more detailed results are given.

We emphasise that this survey is by no means comprehensive. Our goal is the limited one of giving an introduction to the subject, and placing the spotlight on some key themes where there is active work going on. For more systematic study of L\'{e}vy processes on Lie groups and invariant Markov processes on manifolds, see \cite{Liao, LiaoN}. An important topic that is not dealt with here is the application of Malliavin calculus on Wiener--Poisson space to study regularity of transition densities for jump--diffusions. The recently published monograph \cite{Kun} presents a systematic account of key results in this area.

\vspace{5pt}

{\bf Notation.} If $E$ is a locally compact Hausdorff space, the Borel $\sigma$--algebra of $E$ is denoted as ${\mathcal B}(E)$. We denote by ${\mathcal F}(E)$, the space of all real--valued functions on $G$ and $C_{0}(E)$ is the Banach space (with respect to the supremum norm $||\cdot||_{\infty}$) of all real--valued, continuous functions on $G$ that vanish at infinity. If $M$ is a finite--dimensional real $C^{\infty}$--manifold, then $C_{c}^{\infty}(M)$ is the dense linear manifold in $C_{0}(M)$ comprising all smooth functions with compact support.

\noindent The trace of a real or complex $d \times d$ matrix $A$ is written tr$(A)$. We denote as $B_{1}$, the open ball of radius $1$ in $\Rd$ that is centred on the origin.

\section{Markov and Feller Processes}

Let $(\Omega, {\mathcal F}, P)$ be a probability space and $(E, {\mathcal E})$ be a measurable space. An {\it $E$--valued stochastic process} is a family $X: = (X_{t}, t \geq 0)$ of random variables defined on $\Omega$ and taking values in $E$ (so for all $t \geq 0, X_{t}$ is ${\mathcal F}-{\mathcal E}$ measurable). Now suppose that ${\mathcal F}$ is equipped with a filtration $({\mathcal F}_{t}, t \geq 0)$. An adapted $E$--valued stochastic process $X$ is a {\it Markov process} if for all bounded measurable functions $f:E \rightarrow \R$, and all $0 \leq s \leq t < \infty$,
$$ \E(f(X_{t})|{\mathcal F}_{s}) = \E(f(X_{t})|X_{s})~~~\mbox{(a.s.)}$$

We then obtain a family of linear operators $(T_{s,t}, 0 \leq s \leq t < \infty)$ (in fact, these are also contractions and positivity preserving) on the Banach space $B_{\mathcal E}(E)$ of all bounded measurable real--valued functions on $E$ (equipped with the supremum norm) by the prescription
$$ T_{s,t}f(x) = \E(f(X_{t})|X_{s} = x).$$
From now on all Markov processes that are considered will be {\it homogeneous}, in that $T_{s,t} = T_{0,t-s} =:T_{t-s}$, unless otherwise stated. Then the family $(T_{t}, t \geq 0)$ is an algebraic operator (contraction) semigroup on $B_{\mathcal E}(E)$, in that for all $s, t \geq 0$:
$$ T_{s+t} = T_{s}T_{t},~~\mbox{and}~~T_{0} = I.$$
From now on, we will assume that $E$ is a locally compact, second countable Hausdorff space and ${\mathcal E}$ is its Borel $\sigma$--algebra
For each $t \geq 0, B \in {\mathcal E}, x \in E$ we define the transition probability by
$$ p_{t}(x,B) = P(X_{t} \in B| X_{0} = x) = \E({\bf 1}_{X_{t}^{-1}(B)}|X_{0} = x).$$
We then have the representation
\begin{equation} \label{semint}
T_{t}f(x) = \int_{E}f(y)p_{t}(x,dy),
\end{equation}
for all $t \geq 0, f \in B_{\mathcal E}(E), x \in E$, where we have taken a regular version\footnote{Our assumptions on $(E, {\mathcal E})$ ensure that such a version exists.} of the transition probability.
If the mappings $x \rightarrow p_{t}(x,B)$ are measurable, we have the Chapman--Kolmogorov equations
$$ p_{s+t}(x,B) = \int_{E}p_{s}(y, B)p_{t}(x, dy).$$
We would like to write down a differential equation for the transition probabilities which would enable us to extract information about these. This should be of the form of {\it Kolmogorov's forward equation}
\begin{equation} \label{KFE}
\frac{\partial p_{t}(x, B)}{\partial t} = A^{\dagger}p_{t}(x, B ),
\end{equation}
where $A^{\dagger}$ is the formal adjoint of a linear operator $A$ acting on a suitable space of functions. We say that our process $X$ is a {\it Feller process} if $(T_{t}, t \geq 0)$ is a {\it Feller semigroup}, i.e.

\begin{enumerate}
\item $T_{t}(C_{0}(E)) \subseteq C_{0}(E)$ for all $t \geq 0$,
\item $\lim_{t \rightarrow 0}||T_{t}f - f|| = 0$ for all $f \in C_{0}(E).$
\end{enumerate}

It then follows that $(T_{t}, t \geq 0)$ is a strongly continuous contraction semigroup on $C_{0}(E)$. Hence it has an {\it infinitesimal generator} $A$ defined on a dense subspace $D_{A}$ of $C_{0}(E)$ so that for all $f \in D_{A}$,
$$ \lim_{t \rightarrow 0}\left|\left| \frac{T_{t}f - f}{t} - Af\right|\right| = 0.$$
It is precisely this operator that enables us to make sense of (\ref{KFE}).

\section{Stochastic Differential Equations on Manifolds}

There are two ways of making sense of stochastic differential equations (SDEs) on manifolds. The first dates back to It\^{o} \cite{It}, and is nicely described in Chapter 5 of Ikeda--Watanabe \cite{IW}. It involves solving the equation in local co--ordinates in each chart and then showing that the solutions transform geometrically on overlaps. Another method, which can be found in Elworthy \cite{El}, involves using the Whitney or Nash embedding theorem to embed the manifold into a Euclidean space of larger dimension. In that larger space, we must then show that if the initial data lie on the embedded manifold, then so does the solution for all later times. If the driving noise is a continuous semimartingale we must use the Stratonovitch differential to set up the SDE, and for discontinuous semimartingales, the more general Marcus canonical form. However we cannot expect to get Markov processes as solutions with such great generality. Nonetheless, by the argument of section 6.4.2 in \cite{Appbk} pp.387--8, we find that if global solutions exist, then they yield Markov processes when the driving noise is a L\'{e}vy process. To be more specific, we follow \cite{AT}. Let $M$ be a manifold of dimension $d$, and consider an $\Rd$--valued L\'{e}vy process $L = (L(t), t \geq 0)$ having L\'{e}vy measure $\nu$ and L\'{e}vy--It\^{o} decomposition (see Chapter 2 of \cite{Appbk} for background and explanation of the notation)
$$ L_{i}(t) = b_{i}t + \sum_{j=1}^{m}\sigma_{ij}B_{j}(t) + \int_{B_{1}\setminus \{0\}}y_{i}\widetilde{N}(t, dy) + \int_{B_{1}^{c}}y_{i}N(t, dy),$$
for all $i = 1, \ldots, d, t \geq 0$.

Let $Y_{1}, \ldots, Y_{d}$ be $C^{\infty}$ vector fields which have the properties that

\begin{enumerate}
\item[(A1)] All finite linear combinations of $Y_{1}, \ldots, Y_{d}$ are complete,
\item[(A2)] Each $Y_{j}$ has bounded derivatives to all orders in every co--ordinate system (obtained by smoothly embedding $M$ into a Euclidean space).
\end{enumerate}
Now consider the following SDE
\begin{equation} \label{Mark}
dX_{t} = \sum_{i=1}^{d}Y_{i}(X_{t-}) \diamond dL(t)
\end{equation}
with initial condition $X(0) = p$ (a.s.) (where $p \in M$). The $\diamond$ stands for the Markus canonical integral, so in a local co--ordinate system containing $p$ we have the symbolic form

\bean X_{t} & = & p + \sum_{i=1}^{d}b_{i}\int_{0}^{t}Y_{i}(X_{s-})ds + \sum_{i=1}^{d}\sum_{j=1}^{m}\sigma_{ij}\int_{0}^{t}Y_{i}(X_{s-}) \circ dB_{j}(s)\\
& + & \int_{0}^{t}\int_{B_{1}\setminus \{0\}}\left[\exp\left(\sum_{j=1}^{d}y_{j}Y_{j}\right)(X_{s-}) - X_{s-}\right]\widetilde{N}(t, dy)\\
& + & \int_{0}^{t}\int_{B_{1}^{c}}\left[\exp\left(\sum_{j=1}^{d}y_{j}Y_{j}\right)(X_{s-}) - X_{s-}\right]N(t, dy)\\
& + & \int_{0}^{t}\int_{B_{1}\setminus \{0\}}\left[\exp\left(\sum_{j=1}^{d}y_{j}Y_{j}\right)(X_{s-}) - X_{s-} - \sum_{j=1}^{d}y_{j}Y_{j}(X(s-))\right]\nu(dy)ds. \eean

\noindent where $\exp$ is the exponential mapping from complete vector fields to diffeomorphisms, and $\circ$ is the Stratonovitch differential. Under the stated conditions, we do indeed obtain a global solution (that is in fact a stochastic flow of diffeomorphisms) and which, as discussed above will be a Markov process. In fact the conditions are always satisfied if the manifold is compact. We might also ask about the Feller property. In general this is not so easy to answer; for SDEs driven by L\'{e}vy processes in Euclidean space, it is sufficient for coefficients to be bounded as well as suitably Lipschitz continuous (see section 6.7 of \cite{Appbk}). We do have the following result for manifolds.

\begin{theorem} If $M$ is compact then the solution to (\ref{Mark}) is a Feller process. If $A$ is the infinitesimal generator of the transition semigroup, then $C^{2}(M) \subseteq$ Dom$(A)$ and for all $f \in C^{2}(M), x \in M$,
\begin{eqnarray} \label{Mgen}
Af(x) & = & \sum_{i=1}^{d}b_{i}Y_{i}f(x) + \frac{1}{2}\sum_{i=1}^{d}\sum_{j=1}^{m}a_{ij}Y_{i}Y_{j}f(x) \nonumber \\
& + & \int_{\Rd \setminus \{0\}}\left[f\left(\exp\left(\sum_{i=1}^{d}y_{i}Y_{i}\right)x\right) - f(x) - \sum_{i=1}^{d}y_{i}Y_{i}f(x){\bf 1}_{B_{1}(0)}(y)\right]\nu(dy), \nonumber \\
& &
\end{eqnarray}
where the matrix $a: = \sigma \sigma^{T}$.
\end{theorem}

\begin{proof} We just sketch this as it is along similar lines to that of Theorem 6.7.4 in \cite{Appbk} pp.402--3. Let $\Phi_{t}$ be the solution flow that takes each $x \in M$ to the solution of (\ref{Mark}) at time $t$ with initial condition $X_{0} = x$ (a.s.). Then we have $T_{t}f(x) = \E(f(\Phi_{t}(x)))$ for each $t \geq 0, f \in C(M)$. Both (A1) and (A2) hold and so the mapping $\Phi_{t}$ is continuous. Hence $T_{t}:C(M) \rightarrow C(M)$ by dominated convergence. Using It\^{o}'s formula we deduce that for all $f \in C^{2}(M)$,
$$ T_{t}f(x) - f(x) = \int_{0}^{t}T_{s}Af(x)ds,$$
and the result follows easily from here.
\end{proof}

Another interesting example of a class of Feller processes on compact Riemannian manifolds are the isotropic L\'{e}vy processes which are obtained by first solving the equation

\begin{equation} \label{isotrop}
dR_{t} = \sum_{j=1}^{d}H_{j}(R_{t-}) \diamond dL_{j}(t),
\end{equation}

\noindent on $O(M)$, the bundle of orthonormal frames over $M$. Here $H_{1}, \ldots, H_{d}$ are horizontal vector fields (with respect to the Riemannian connection on $M$), and $L$ is an isotropic L\'{e}vy process on $\Rd$ (i.e. its laws are $O(d)$--invariant). The required isotropic process is given by $X_{t} = \pi(R_{t})$ where $\pi:O(M) \rightarrow M$ is the canonical surjection. The generator of the Feller process $X$ is the sum of a non--negative multiple of the Laplace--Beltrami operator and an integral superposition of jumps along geodesics weighted by the L\'{e}vy measure of $L$. For details see \cite{AE}. If $L$ is a standard Brownian motion so that (\ref{isotrop}) reduces to a Stratonovitch equation, then $X$ is Brownian motion on $M$ whose generator is the Laplace--Beltrami operator.

Now return to the SDE (\ref{Mark}). In \cite{AK}, it is shown the this equation has a unique solution (which will be a Markov process) if (A1) and (A2) are replaced by the single condition that the vector fields $Y_{1}, \ldots, Y_{d}$ generate a finite--dimensional Lie algebra. In general this is quite a strong assumption; as we are essentially saying that solution flow only explores a ``small'' finite--dimensional part of the ``huge'' infinite--dimensional diffeomorphism group of $M$. But one case where it occurs naturally is if $M$ is a Lie group, which we now denote as $G$, and if $Y_{1}, \ldots, Y_{d}$ are assumed to be left--invariant vector fields which form a basis for the Lie algebra $\g$ of $G$. Before proceeding further, we will define a L\'{e}vy process on a Lie group. This is precisely as in the well--known Euclidean space, i.e. a \cadlag~process $X = (X_{t}, t \geq 0)$ , that is stochastically continuous, satisfies $X_{0} = e$ (a.s.), where $e$ is the neutral element of $G$ and has stationary and independent increments, where the increment between times $s$ and $t$ with $s \leq t$, is understood to be $X(s)^{-1}X(t)$. Since the pioneering work of Hunt in 1956 \cite{Hunt}, it has been known that L\'{e}vy processes are Feller processes (take the filtration to be the natural one coming from the process). The associated Feller semigroup is defined for $f \in C_{0}(G), \sigma \in G, t \geq 0$ by

$$ T_{t}f(\sigma) = \E(f(\sigma X_{t})) = \int_{G}f(\sigma \tau)\rho_{t}(d \tau),$$ where $\rho_{t}$ is the law of $X_{t}$. In fact $(\rho_{t}, t \geq 0)$ is a weakly continuous, convolution semigroup of probability measures on $(G, {\mathcal B}(G))$. Hunt also showed that the generator $A$ has the following canonical representation on $C_{c}^{\infty}(G) \subseteq \mbox{Dom}(A)$:
For each
$\sigma \in G, f \in C_{c}^{\infty}(G)$,
\begin{eqnarray} \label{hu}
Af(\sigma) & = & \sum_{i=1}^{d}b_{i}Y_{i}f(\sigma) +
\frac{1}{2}\sum_{i,j=1}^{d}a_{ij}Y_{i}Y_{j}f(\sigma)\nonumber \\
 & + & \int_{G \setminus \{e\}}\left(f(\sigma \tau) - f(\sigma) -
   \sum_{i=1}^{d}x_{i}(\tau)Y_{i}f(\sigma)\right)\nu(d\tau), \nonumber \\
   & &
\end{eqnarray}
where $b = (b_{1}, \ldots, b_{d}) \in {\R}^{d}, a = (a_{ij})$ is a
non-negative definite, symmetric $d \times d$ real-valued matrix
and $\nu$ is a L\'{e}vy measure on $G$, i.e. a Borel measure on $G \setminus \{e\}$ which is such that for every canonical co--ordinate neighbourhood $U$ of $e$ we have
	\begin{equation*}
	\int_{U \setminus \{e\}} \left(\sum_{i =1}^d x_i^2(g)\right)\nu(dg) < \infty, \; \text{and } \nu(U^c) < \infty.
 \end{equation*} Here $x_{i} \in C_{c}^{\infty}(G)$ for $i = 1, \ldots, d$ are such that $(x_{1}, \ldots, x_{d})$ are canonical co--ordinates for $G$ in the neighbourhood $U$ of $e$. The triple $(b, a, \nu)$ are called the characteristics of $X$, and they uniquely determine $(\rho_{t}, t \geq 0)$.

In the case where the exponential map from $\g$ to $G$ is onto (e.g. if $G$ is compact and connected, or simply connected and nilpotent), then the L\'{e}vy process $X$ arises as the unique solution to the SDE (\ref{Mark}), provided the characteristics of the driving L\'{e}vy process $L$ are chosen so as to match those in (\ref{hu}). More generally, $X$ is written as a more general type of stochastic integral equation (for details see \cite{AK}, \cite{Liao}).

\vspace{5pt}

{\bf Remark}. A more general form of SDE on manifold than (\ref{Mark}) was developed by S. Cohen \cite{Coh}. In this case, the driving noise itself takes values in a manifold, which is in general different to that in which the solution flow takes values. After some years of neglect, this theory has recently found some new applications in work of \cite{AVMU}.

\section{The Positive Maximum Principle and Courr\`{e}ge Theory}

Let $E$ be a locally compact Hausdorff space and $A$ be a linear mapping with domain Dom$(A) \subseteq C_{0}(E)$ and range Ran$(A) \subseteq {\mathcal F}(E)$. Let $D_{A}$ be a linear subspace of Dom$(A)$. We say that $A$ satisfies the positive maximum principle (PMP) with respect to $D_{A}$ if $f \in D_{A}$ and $f(y) = \sup_{x \in E}f(x) \geq 0 \Rightarrow Af(y) \leq 0$. It is usual in the literature (see e.g. \cite{EK}) to assume that $A$ maps continuously i.e. that Ran$(A) \subseteq C_{0}(E)$. Then one can prove powerful results - e.g. that $A$ is dissipative, and hence if Dom$(A)$ is dense, then $A$ is closeable (see \cite{EK}, Chapter 4, Lemma 2.1 (p.165) and Chapter 1, Lemma 2.11 (p.10)). We say that $A$ has the full PMP if $D_{A} = $Dom$(A)$.

 The connection with our work is through the following well--known result.

\begin{prop} If $A$ is the generator of a Feller semigroup in $C_{0}(E)$ then it has the full PMP.
\end{prop}

\begin{proof} Using (\ref{semint}) we have for all $f \in$ Dom$(A)$ with $f(y) = \sup_{x \in E}f(x)$,
\bean Af(y) & = & \lim_{t \rightarrow 0}\frac{T_{t}f(y) - f(y)}{t}\\
& = & \lim_{t \rightarrow 0}\frac{1}{t}\int_{E}(f(x) - f(y))p_{t}(y, dx) \leq 0. \eean
\end{proof}

There are stronger results than this. The celebrated {\it Hille--Yosida--Ray theorem} states that a linear operator $A$ is the generator of a positivity preserving contraction semigroup $(T_{t}, t \geq 0)$ if and only if $A$ is densely defined, closed, satisfies the full PMP and $\lambda I - A$ is onto $C_{0}(E)$ for all $\lambda > 0$. If such a semigroup is also {\it conservative}, i.e. it has an extension to the space of bounded measurable functions on $E$ such that $T_{t}1 = 1$ for all $t \geq 0$, then we may effectively use Kolomogorov's construction to obtain a Feller process for which $(T_{t}, t \geq 0)$ is the transition semigroup. For details see \cite{BSW} pp.13--238 and \cite{EK} pp.165--73.

\subsection{The Courr\`{e}ge Theorem on Euclidean Space and Manifolds}

In this subsection we take $D_{A} = C_{c}^{\infty}(\Rd)$ and consider linear operators $A$ that satisfy the PMP. The following key classification result was first published by Courr\`{e}ge in 1965. In the following statement, we give the form (\ref{Courrth1}) as can be found in Hoh \cite{Hoh}.

\begin{theorem} \label{Courrth} [Courr\`{e}ge theorem]
Let $A$ be a linear operator from $C_{c}^{\infty}(\Rd)$ to $\mathbb{F}(\Rd)$. Then $A$ satisfies the PMP if and only if there exists a unique quadruple $(a(\cdot), b(\cdot), c(\cdot), \nu(\cdot))$ wherein for all $x \in \Rd$, $a(x)$ is a $d \times d$ non--negative definite symmetric matrix, $b(x)$ is a vector in $\Rd, c(x)$ is a non--negative constant and $\nu(x, \cdot)$ is a L\'{e}vy measure on $\Rd$, so that for all $f \in C_{c}^{\infty}(\Rd)$
\begin{eqnarray} \label{Courrth1}
Af(x) & = & -c(x)f(x) + \sum_{i=1}^{d}b_{i}(x)\partial_{i}f(x) + \sum_{i,j =1}^{d}a_{ij}(x)\partial_{i}\partial_{j}f(x) \nonumber \\
& + & \int_{\Rd \setminus \{0\}}\left(f(x + y) - f(x) - \sum_{i=1}^{d}y_{i}\partial_{i}f(x){\bf 1}_{B_{1}}(y)\right)\nu(x,dy) \nonumber \\
& &
\end{eqnarray}
\end{theorem}

In the same paper, Courr\`{e}ge proved that $A$ is a pseudo--differential operator:
$$ Af(x) = \int_{\Rd}e^{i x \cdot y}\eta(x, y)\widehat{f}(y)dy,$$
where $\widehat{f}$ denotes the usual Fourier transform: $\widehat{f}(y) = \frac{1}{(2\pi)^{d/2}}\int_{\Rd}e^{-i x \cdot y}f(x)dx$, and $\eta$ is the symbol of the operator where

\begin{eqnarray} \label{symbol}
\eta(x,y) & = & - c(x) + i b(x) \cdot y - a(x)y \cdot y  \nonumber \\
& + & \int_{\Rd \setminus \{0\}}(e^{i x \cdot y} -1  - i x \cdot y {\bf 1}_{B_{1}}(y))\nu(x,dy), \end{eqnarray}
so that formally:

\begin{equation} \label{form}
\eta(x,y) = e^{-i x \cdot y}A(e^{i x \cdot y}).
\end{equation}

If $A$ is the generator of a L\'{e}vy process in $\Rd$, then $c = 0, b, a$ and $\nu$ are independent of the value of $x \in \Rd$, (\ref{Courrth1}) is the Euclidean version of (\ref{hu}) and $\eta$ is the characteristic exponent whose form is that of the classical L\'{e}vy--Khintchine formula (see e.g. Chapter 3 of \cite{Appbk} and also section 5 below). The general result (\ref{symbol}) is important as it is a valuable source of probabilistic information about rich Feller processes, i.e. those for which $C_{c}^{\infty}(\Rd) \subseteq$ Dom$(A)$, and so come under the auspices of this theory. See \cite{BSW} and references therein for details.

The Courr\`{e}ge theorem has been generalised to manifolds, first by Courr\`{e}ge in the compact case \cite{Courr1} and then by Courr\`{e}ge, Bony and Prioret \cite{BCP} in the general case. These authors succeeded in showing that if $A:C_{c}^{\infty}(M) \rightarrow {\mathcal F}(M)$ satisfies the PMP, then it has a decomposition of similar form to (\ref{Courrth1}) relative to a local chart at a point. So they were able to describe the form of the operator in local co--ordinate systems. If the manifold is a Lie group or a symmetric space, we can obtain a global formalism.

\subsection{The Courr\`{e}ge Theorem on Lie Groups}

We summarise some of the main results of \cite{AT1}. Let $G$ be a Lie group with Lie algebra $\g$. Here the strategy is to imitate the approach in \cite{Hoh}, but with the vector fields $\{\partial_{1}, \ldots, \partial_{d}\}$ in Euclidean space replaced by the Lie algebra basis $\{Y_{1}, \ldots, Y_{d}\}$ for $\g$. We say that a linear functional $T: C_{c}^{\infty}(G) \rightarrow \R$ satisfies the {\it positive maximum principle} (PMP), if whenever $f \in C_{c}^{\infty}(G)$ with $f(e) = \sup_{\tau \in G}f(\tau) \geq 0$ then $Tf \leq 0$. Then the linear operator $A$ satisfies the PMP if and only if each of the linear functionals $A_{g}$ satisfy the PMP, where $A_{g}f:=A(L_{g^{-1}}f)(g)$ for each $g \in G, f \in C_{c}^{\infty}(G)$. Here $L_{g}$ is the usual left translation defined by $L_{g}f(\sigma) = f(g\sigma)$ for $\sigma \in G$. We can recover the action of $A$ from that of the $A_{g}'s$ by $Af(g) = A_{g}L_{g}f$.
So the problem is now reduced to studying the PMP for linear functionals $T$. Now it can be shown that  if $T$ satisfies the PMP, then it is  a distribution of order $2$. This is not a distribution in the usual sense. It is defined in the same way, but with the role of each partial derivative $\partial_{i}$ replaced by $Y_{i}$ for $i = 1, \ldots, d$.
	
We can then prove the first important result
\begin{theorem} \label{PMPLF}
	Let $T:C_c^\infty(G)\to \R$ be a linear functional satisfying the PMP. Then there exists $c \geq 0, b = (b_{1}, \ldots, b_{d}) \in \Rd$, a non--negative definite symmetric $d \times d$ real--valued matrix $a = (a_{ij})$, and a L\'{e}vy measure $\mu$ on $G$ such that for all $f \in C_c^\infty(G)$,
	\begin{eqnarray} \label{distr_form}
	Tf & = & \sum_{i,j =1}^d a_{ij}Y_iY_jf(e) + \sum_{i=1}^d b_i Y_if(e) - cf(e)\nonumber \\ & + & \int_{G\{e\}} \left(f(g)-f(e) - \sum_{i=1}^d x_i(g) Y_if(e)\right) \mu(dg),
\end{eqnarray}
	\end{theorem}

The proof involves using a positivity argument (essentially the Riesz lemma) to show that there exists a Borel measure $\mu$ on $G \setminus \{e\}$ so that for all $f \in C^\infty_c(G\setminus \{e\})$
	\begin{equation}\label{f_mu}
 Tf = \int_{G\setminus \{e\}}f(g)\mu(dg).
	\end{equation}
It then turns out that $\mu$ is a L\'{e}vy measure. Next we introduce a linear functional $S: C_c^\infty(G) \to \R$, by
	\begin{equation*}
	Sf := \int_{G \setminus \{e\}}\left[f(g) - f(e)-\sum_{i=1}^d x_i(g)Y_i f(e)\right] \mu(dg).
	\end{equation*}

Then $S$ satisfies the PMP and so is a distribution of order $2$. Hence so is $P = T-S$. But supp$(P) \subseteq \{e\}$ and so $P$ must take the form

\begin{equation*}
	Pf = \sum_{i,j=1}^d a_{ij}Y_iY_jf(e) +\sum_{i=1}^d b_i Y_{i}f(e) - cf(e).
\end{equation*}

It remains to prove that the matrix $(a_{ij})$ is positive definite and that $c \geq 0$. For this we refer the reader to the paper \cite{AT1}. Once Theorem \ref{PMPLF} is established, we can use the ``left translation trick'' described above to get the main result:

\begin{theorem} \label{PMP1} The mapping $A: C_{c}^{\infty}(G) \rightarrow {\mathcal F}(G)$ satisfies the PMP if and only if there exist functions $c, b_{i}, a_{jk}~(1 \leq i,j,k \leq d)$ from $G$ to $\R$, wherein $c$ is non--negative, and the matrix $a(\sigma): = (a_{jk}(\sigma))$ is non--negative definite and symmetric for all $\sigma \in G$, and a L\'{e}vy kernel\footnote{We say that $\mu$ is a {\it L\'{e}vy kernel} if $\mu(\sigma, \cdot)$ is a L\'{e}vy measure for all $\sigma \in G$.} $\mu$, such that for all $f \in C_{c}^{\infty}(G), \sigma \in G$,
\begin{eqnarray} \label{PMP2}
Af(\sigma) & = & -c(\sigma)f(\sigma) + \sum_{i=1}^{d}b_{i}(\sigma)Y_{i}f(\sigma) + \sum_{j,k = 1}^{d}a_{jk}(\sigma)Y_{j}Y_{k}f(\sigma) \nonumber \\
& + & \int_{G \setminus \{e\}}\left(f(\sigma \tau) - f(\sigma) - \sum_{i=1}^{d}x_{i}(\tau)Y_{i}f(\sigma)\right)\mu(\sigma, d\tau).
\end{eqnarray}
\end{theorem}

In \cite{AT1} sufficient conditions are imposed on the coefficients to ensure that $A: C_{c}^{\infty}(G) \rightarrow C_{0}(G)$, and we will assume that these hold from now on.

We can represent linear operators satisfying the PMP as pseudo--differential operators when $G$ is compact. To carry this out we need the unitary dual $\G$ comprising equivalence classes (with respect to unitary equivalence) of all irreducible representations of $\G$ in some complex (finite--dimensional) Hilbert space. So if $\pi \in \G$, then for each $g \in G, \pi(g)$ is a unitary matrix acting in a space $V_{\pi}$ which has dimension $d_{\pi}$. The Fourier transform $\widehat{f}$ of $f \in C_{c}^{\infty}(G)$ is the matrix--valued function on $\G$ defined by
$$ \widehat{f}(\pi) = \int_{G}f(g)\pi(g^{-1})dg,$$
where integration is with respect to normalised Haar measure on $G$. The Ruzhansky--Turunen theory of pseudo--differential operators \cite{RT}, starts from the Fourier inversion formula (just as in the classical case):
$$ f(g) = \sum_{\pi \in \G}d_{\pi}\tr(\widehat{f}(\pi)\pi(g)).$$
Then we say that $A: C_{c}^{\infty}(G) \rightarrow C_{0}(G)$ is a pseudo--differential operator with matrix valued symbol $j_{A}(\sigma, \pi)$ acting in $V_{\pi}$ for each $\sigma \in G, \pi \in \G$ if

$$ Af(g) = \sum_{\pi \in \G}d_{\pi}\tr(j_{A}(\sigma, \pi)\widehat{f}(\pi)\pi(g)).$$

In \cite{App3} it was shown that the generators of  L\'{e}vy processes on compact Lie groups are pseudo differential operators, and this was further extended to some classes of Feller processes in \cite{App2}. Now we have the more general result:

\begin{prop} \label{PDo1} If $A: C_{c}^{\infty}(G) \rightarrow C_{0}(G)$ satisfies the PMP then it is a pseudo--differential operator with symbol
\begin{eqnarray} \label{PDo2}
j_{A}(\sigma, \pi) & = & -c(\sigma)I_{\pi} + \sum_{i=1}^{d}b_{i}(\sigma)d\pi(Y_{i}) + \sum_{j,k = 1}^{d}a_{jk}(\sigma)d\pi(Y_{j})d\pi(Y_{k}) \nonumber \\
& + & \int_{G \setminus \{e\}}\left(\pi(\tau) - I_{\pi} - \sum_{i=1}^{d}x_{i}(\tau)d\pi(Y_{i})\right)\mu(\sigma, d\tau),
\end{eqnarray}
for each $\pi \in \G, \sigma \in G$.
\end{prop}
Here if $Y \in \g, d\pi(Y)$ is the skew--hermitian matrix acting in $V_{\pi}$ which is uniquely defined by
$$ \pi(\exp(tY)) = e^{t d\pi(Y)},$$
for all $t \geq 0$. In the case where $c = 0$ and the characteristics $b, a$ and $\mu$ are independent of $\sigma \in G$, (\ref{PDo2}) coincides with the form of the L\'{e}vy--Khinchine formula on compact Lie groups (see section 5.5 in \cite{Applebaumbk2}).

Note that the analogue of (\ref{form}) is
$$j_{A}(\sigma, \pi) = \pi(\sigma^{-1})A\pi(\sigma),$$
where $A\pi(\sigma)$ is the matrix $(A\pi_{ij}(\sigma))$, and since $G$ is compact and $\pi_{ij}(\cdot) \in C^{\infty}(G)$ for all $i, j = 1, \ldots d_{\pi}$, this identity is mathematically rigorous.

In \cite{AT2} these ideas are extended to study linear operators satisfying the PMP on a symmetric space $M$. Observe that $M$ may be identified with the homogeneous space $G/K$, where $G$ is the isometry group of $M$ (which turns out to be a Lie group), and $K$ is the compact subgroup of $G$ comprising isometries leaving some given point fixed\footnote{It doesn't matter which point we choose.}. This enables us to use group theoretic techniques as described above, but with additional symmetries arising from the action of $K$. A feature of this theory is that, by using the theory of spherical functions, we can study a class of linear operators on $M$ that satisfy the PMP and are pseudo--differential operators with scalar--valued symbols.

\section{Invariant Markov Processes}

Let $M$ be a manifold under the action of a Lie group $G$, and let $X=(X_t$; $t\geq 0)$ be a Markov process in $M$ with transition semigroup $T_t$ as defined in Section~2. For simplicity,
we write $X_t$ for the process from now on. It should be clear from the context whether $X_t$ is the whole process or just the random variable for a fixed time $t$. The Markov process $X_t$ will be called invariant
under the action of $G$, or $G$-invariant for short, if for any bounded Borel function $f$ on $M$ and $g\in G$,
\begin{equation} \label{inv}
T_t(f\circ g)(x) = (T_tf)(gx), \quad x\in M.
\end{equation}
It is easy to show, see \cite[Proposition~1.1]{LiaoN}, that the $G$-invariance of the Markov process $X_t$ may also be characterized probabilistically as follows: We may think a Markov process $X_t$ as a family of processes,
one for each starting point $x\in M$, all governed probabilistically by the same transition semigroup. If we denote $X_t^z$ for the Markov process starting from $z\in M$, then $X_t$ is $G$-invariant if and only if
for any $g\in G$, $gX_t^z$ and $X_t^{gz}$ are equal in distribution as processes.

The $G$-invariance may also be defined for inhomogeneous Markov processes if $T_t$ in (\ref{inv}) is replaced by $T_{s,t}$.

When $M=\R^d$ and $G=\R^d$ acts as translations on $\R^d$, it is well known that a \cadlag\ Markov process $X_t$ in $\R^d$ is $G$-invariant, or translation-invariant, if and only if it has independent and stationary increments $X_t-X_s$
for $s<t$, that is, if it is a L\'{e}vy process in $\R^d$. The celebrated L\'{e}vy-Khintchine formula gives the characteristic function of each random variable within  a L\'{e}vy process $X_t$:
\[E(e^{i X_t \cdot y }) = e^{t\psi(y)}\]
for any $y=(y_1,\ldots,y_d)\in\R^d$ by
\[\psi(y) = i\sum_{j=1}^d b_jy_j + \frac{1}{2}\sum_{j,k=1}^da_{jk}y_jy_k + \int_{\Rd\setminus\{0\}}(e^{i  z \cdot y } - 1 - i  z \cdot y 1_{B_1}(z))\nu(dz).\]
It provides a useful representation for a L\'{e}vy process in terms of a triple $(b,a,\nu)$, comprising a drift vector $b$, a covariance matrix $a=\{a_{ij}\}$ and a L\'{e}vy measure $\nu$
in the sense that its probability distribution is determined by the triple, and to any such triple, there is an associated L\'{e}vy process, unique in distribution, which also determines the triple. Note that (as we would expect) $\psi$ coincides
with $\eta$, as appears in (\ref{symbol}), when the latter function is constant in the $x$-variable.

In the rest of paper, we will present some results on the more general invariant Markov processes under actions of Lie groups. The reader is referred to \cite{LiaoN} for more details. 

Because a Lie group $G$ acts on itself by left translations $l_g$: $G\to G$, $x\mapsto gx$, for $g\in G$, we may first consider a Markov process $X_t$ in $G$ invariant under this $G$-action. This is a direct extension
of a translation-invariant Markov process in $\R^d$ to a Lie group, and may be identified with a  L\'{e}vy process in $G$ as defined in Section~3, which is characterized by independent and stationary increments
of the form $X_s^{-1}X_t$ for $s<t$. As already mentioned there, a  L\'{e}vy process $X_t$ in $G$ is a Feller process, and
its generator $L$ has an explicit expression given by Hunt's formula (\ref{hu}), expressed in terms of a drift vector $b$, a covariance matrix $a$ and a  L\'{e}vy measure $\nu$. Thus, a  L\'{e}vy process in a Lie group
is represented by a triple $(b,a,\nu)$ just as for a  L\'{e}vy process in $\R^d$.

More generally, we may consider an invariant Markov process $X_t$ in a manifold under the transitive action of a Lie group $G$. Here a transitive action means that for any two points $x$ and $y$ in $M$, there is $g\in G$ such that $gx=y$.
Fix a point $o\in M$ and let $K=\{g\in G$; $go=o\}$. Then $K$ is a closed subgroup $G$, called the isotropy subgroup at point $o$. The manifold $M$ may be identified with the homogeneous space $G/K$, the space of left cosets $gK$, $g\in G$,
via the map $gK\mapsto go$, and the $G$-action on $M$ corresponds to the natural $G$-action on $G/K$ given by $(g,hK)\mapsto ghK$ for $g,h\in G$. Thus, a $G$-invariant Markov process in $M$ may be naturally identified
with a $G$-invariant Markov process in the homogeneous space $G/K$.

A $G$-invariant Markov processes in $G/K$ may also be characterised by independent and stationary increments as for a Markov process in $G$ invariant under left translations, but to state this precisely will require
a little preparation.

Let $M=G/K$ and let $\pi$: $G\to M$ be the natural projection $g\mapsto gK$. A Borel map $S$: $M\to G$ will be called a section map if
\[\pi\circ S = {\rm id}_M \ \mbox{(the identity map on $M$)}.\]
A section map always exists, but is not unique. It may not be smooth on $M$, but can be chosen to be smooth near any point in $M$.

In general, there is no natural product structure on $M=G/K$, but we may define $xy=S(x)y$ with the choice of a section map $S$. Although this definition depends on the choice of $S$,
if $Z$ is a random variable in $M$ that has a $K$-invariant distribution, that is, if $kZ$ is equal to $Z$ in distribution for any $k\in K$, then the distribution of $xZ$ does not depend on the choice of $S$. The increment of
a process $X_t$ in $M=G/K$ over the time interval $(s,\,t]$ is defined by $X_s^{-1}X_t = S(X_s)^{-1}X_t$, and its distribution does not depends on $S$ when $X_t$ is a $G$-invariant Markov process.

\begin{prop} \label{princre}
A process $X_t$ in $M=G/K$ with natural filtration $\{{\mathcal F}_t\}$ is a $G$-invariant Markov process if and only if it has independent and stationary increments in the sense that
for any $s<t$ and any section map $S$, $X_{s,t}=S(X_s)^{-1}X_t$ is independent of $\mathcal{F}_s$ and its distribution depends only on $t-s$, not on $S$.
\end{prop}

The easy proof of the above proposition may be found in \cite[Theorem~1.14]{LiaoN}. Because of this result, a $G$-invariant Markov process in $M=G/K$ will be called a L\'{e}vy process.

When $K$ is compact, a L\'{e}vy process in $M=G/K$ is a Feller process and its generator is determined by Hunt \cite{Hunt} to have essentially the same form as the generator of a L\'{e}vy process in $G$. To state this precisely will require some additional preparation.

For any $g\in G$, the conjugation map $G\to G$, given by $x\mapsto gxg^{-1}$, fixes the identity element $e$ of $G$. Its differential map at $e$ is a linear bijection from the Lie algebra $\g$ of $G$ to $\g$, denoted
as $\Ad(g)$. Let $\p$ be an $\Ad(K)$-invariant subspace of $\g$, complementary to the Lie algebra $\fk$ of $K$. The linear space $\p$ may be identified with the tangent space $T_oM$ of $M$ at $o$ via the natural
projection $\pi$: $G\to M$. Recall that we have chosen a basis $\{Y_1,\ldots,Y_d\}$ of $\g$.
We may assume $Y_1,\ldots,Y_n$ form a basis of $\p$ and $Y_{n+1},\ldots,Y_d$ form a basis of $\fk$, where $n=\dim(M)$. Recall the coordinate functions $x_1,\ldots,x_d\in C_c^\infty(G)$ associated to the
basis $\{Y_1,\ldots,Y_d\}$ of $\g$ as introduced in Section~3. They may be chosen to satisfy $g=\exp[\sum_{i=1}^nx_i(g)Y_i]\exp[\sum_{i=n+1}^dx_i(g)Y_i]$ for $g$ in a neighborhood of $e$, where $\exp$: $\g\to G$
is the Lie group exponential map.
We now define coordinate functions $y_1,\ldots,y_n$ on $M$ associated to the basis $\{Y_1,\ldots,Y_n\}$ of $\p$ to be functions in $C_c^\infty(M)$ such that $\sigma=\pi\{\exp[\sum_{i=1}^ny_i(\sigma)Y_i]\}$ for $\sigma$
in a neighborhood of $o$. They may be chosen to satisfy $\sum_{i=1}^ny_i(\sigma)\Ad(k)Y_i=\sum_{i=1}^ny_i(k\sigma)Y_i$ for all $\sigma\in M$ and $k\in K$ (see
\cite[Section~3.1]{LiaoN}). Note that $y_i=x_i\circ\pi$ for $1\leq i\leq n$.

For $k\in K$, let $[\Ad(k)]$ denote the matrix representing $\Ad(k)$; $\p\to\p$ under the basis of $\{Y_1,\ldots,Y_n\}$, that is, $\Ad(k)\xi_j=\sum_{i=1}^n[\Ad(k)]_{ij}Y_i$. A vector $b=
(b_1,\ldots,b_n)\in\R^n$ will called $\Ad(K)$-invariant if $[\Ad(k)]b=b$ for $k\in K$ and an $n\times n$ non-negative definite real symmetric matrix $a$ will be called $\Ad(K)$-invariant
if $[\Ad(k)]a[\Ad(k)]'=a$ for $k\in K$, where the prime denotes the matrix transpose. A measure $\mu$ on $M$ is called $K$-invariant if $k\mu=\mu$ for $k\in K$, where $k\mu$
is the measure $k\mu(B)=\mu(k^{-1}(B))$. Note that for a $K$-invariant measure $\mu$, the value of the integral
\[\int_M f(xy)\mu(dy) = \int f(S(x)y)\mu(dy), \quad x\in M,\]
does not depend on the choice of the section map $S$. A L\'{e}vy measure on $M$ is defined in the same way as a L\'{e}vy measure on $G$, given in Section~3, but using the coordinate functions on $M$ and
a neighborhood $U$ of $o$, and with the additional requirement that it is $K$-invariant.

The following result is Hunt's generator formula on $M=G/K$. The present form is taken from
\cite{LiaoN}. By this result, a L\'{e}vy process in $M=G/K$ is represented, in distribution, by a triple $(b,a,\nu)$ comprising an $\Ad(K)$-invariant vector $b$, an $\Ad(K)$-invariant matrix $a$
and a L\'{e}vy measure $\nu$ on $M$, just as for a L\'{e}vy process in $\R^d$.

\begin{theorem} \label{huth}
A L\'{e}vy process in $M=G/K$ is a Feller process. Let $A$ be its generator.  Then the domain of $A$ contains $C_c^\infty(M)$, and there is a unique triple $(b,a,\nu)$ as above such that $A$ restricted to $C_c^\infty(M)$
is given by (\ref{hu}) with $d$ replaced by $n$, $x_i$ by $y_i$ and $G\backslash\{e\}$ by $M\backslash\{o\}$. Moreover, given any triple $(b,a,\nu)$, there is a L\'{e}vy process in $M$ starting at $o$, unique in distribution, such that
its generator has the above expression.
\end{theorem}

Because the action of $K$ on $M$ fixes the point $o$, it induces a linear action of $K$ on the tangent space $T_oM$. The homogeneous space $M=G/K$ is called irreducible if $T_oM$ has no proper subspace that
is invariant under this $K$-action. For example, the sphere $S^{n-1}=O(n)/O(n-1)$ is irreducible as a homogeneous space of the orthogonal group $O(n)$ on $\R^n$.

On an irreducible homogeneous space $M=G/K$, it can be shown (see \cite[Section~3.2]{LiaoN}) that under a suitable choice of the basis $\{Y_1,\ldots,Y_n\}$ of $\p$, any $\Ad(K)$-invariant matrix is proportional to the
identity matrix $I$, and when $\dim(M)>1$, there is no nonzero $\Ad(K)$-invariant vector. In this case, Hunt's generator formula (\ref{hu}) on $M$ takes the following simpler form: For $f\in C_c^\infty(M)$,
\begin{equation} \label{hu2}
Af(\sigma) = \frac{\alpha}{2}\sum_{i=1}^nY_iY_if(\sigma) + \int_G[f(\sigma\tau) - f(\sigma)]\nu(d\tau),
\end{equation}
where the integral is understood as the principal value, that is, as the limit of $\int_{U^c}[f(\sigma\tau)-f(\sigma)]\nu(d\tau)$ when the $K$-invariant neighborhood $U$ of $o$ shrinks to $o$. Therefore,
on an irreducible homogeneous space $M=G/K$ with $\dim(M)>1$, a  L\'{e}vy process is represented, in distribution, by a pair $(\alpha,\nu)$ comprising a real number $\alpha\geq 0$ and a L\'{e}vy measure $\nu$ on $M$.

We note that the suitable choice of the basis $\{Y_1,\ldots,Y_n)$ of $\p$ means that it is chosen to be orthonormal under an $\Ad(K)$-invariant inner product on $\p$.

Recall that the L\'{e}vy-Khintchine formula gives the characteristic function of a L\'{e}vy process in $\R^d$. Such a formula may be extended to a general Lie group
or a homogeneous space. See \cite[Section~5.5]{Applebaumbk2} and \cite{Heyer} for its extension to a compact Lie group in terms of group representations, and \cite{Applebaum2} for the extension to a general Lie group. In these generalisations, the role of the characteristic exponent is played by a matrix--valued function on the unitary dual (in the compact case), or a function whose values are linear operators in a Hilbert space (in the general case).
On a symmetric space, which is a special type of homogeneous space, Gangolli \cite{Gangolli} obtained a L\'{e}vy-Khintchine type formula, in terms of the spherical transform, that closely resembles the classical form in that the characteristic exponent is scalar--valued. A simpler approach to Gangolli's result based on Hunt's generator formula may be found in \cite{Applebaum} and \cite{LiaoWang}. See also \cite[Chapter~5]{LiaoN} for a more systematic discussion.

We have considered invariant Markov processes under a transitive action of a Lie group, and have identified such processes with L\'{e}vy processes, which are characterized by independent and stationary increments.
We will next consider the non-transitive action of a Lie group $G$. In this case, the state space $M$ is a collection of $G$-orbits, and an invariant Markov process may be decomposed into a ``radial'' part, that
is transversal to $G$-orbits, and an ``angular'' part, that lives in a $G$-orbit and is a $G$-invariant inhomogeneous Markov process. Such a process may be characterized by independent, but not necessarily stationary, increments, and so will be called an inhomogeneous L\'{e}vy process.
Before we discuss the decomposition of an invariant Markov process under a non-transitive action, we will first briefly discuss inhomogeneous L\'{e}vy processes in Lie groups
and homogeneous spaces.

\section{Inhomogeneous L\'{e}vy Processes}

Let $G$ be a Lie group and $K$ be a compact subgroup as before. A \cadlag\ inhomogeneous Markov process in $G$ or in $M=G/K$, that is invariant under the left translation in $G$ or $G$-invariant in $M$, will be called
an inhomogeneous L\'{e}vy process. This definition is justified by the following easy proposition (see \cite[Theorem~1.24]{LiaoN}).

\begin{prop} \label{prindinc}
Let $X_t$ be \cadlag\ process in $G$ or in $M=G/K$ with natural filtration $\{{\mathcal F}_t\}$. Then $X_t$ is an inhomogeneous L\'{e}vy process if and only if it has independent increments in the sense that for
any $s<t$, $X_s^{-1}X_t$ is independent of ${\mathcal F}_s$, where $X_s^{-1}X_t$ is understood as $S(X_s)^{-1}X_t$ on $M$ for any section map $S$ and its distribution is assumed not to depend on $S$.
\end{prop}

Recall that a L\'{e}vy process in $G$ is a Feller process and the domain of its generator $A$ contains $C_c^\infty(G)$. By a standard result in Markov process theory (see Proposition~1.7
in \cite[Chapter~4]{EK}), for any $f\in C_c^\infty(G)$,
\[f(X_t) - \int_0^t Af(X_s)ds\]
is a martingale under the natural filtration of $X_t$, where the generator $A$, given by (\ref{hu}), is expressed in terms of a triple $(b,a,\nu)$. This martingale property in fact provides a complete representation
of the L\'{e}vy process $X_t$, in distribution, in the sense that given any triple $(b,a,\nu)$, there is a L\'{e}vy process, unique in distribution, having this martingale property.

A \cadlag\ process $X_t$ is called stochastically continuous if $P(X_{t-}=X_t)=1$ for any $t>0$. A L\'{e}vy process is stochastically continuous, but an inhomogeneous L\'{e}vy process may not be. Feinsilver \cite{Feinsilver}
obtained a martingale representation for a stochastically continuous inhomogeneous L\'{e}vy process in a Lie group $G$ by a time-dependent triple. Such a representation holds also for inhomogeneous L\'{e}vy processes
in a homogeneous space $G/K$, even for non-stochastically continuous ones, see \cite{LiaoN}. The general representation is rather complicated, so we will not discuss it here, but will present this representation on
an irreducible $G/K$ where it takes a very simple form.

A family of L\'{e}vy measures $\nu(t,\cdot)$, indexed by time $t\geq 0$, will be called a L\'{e}vy measure function if it is nondecreasing and continuous under weak convergence. It will be called an extended
L\'{e}vy measure function if it may not be continuous, but is \cadlag, and $\nu(t,M)-\nu(t-,M)\leq 1$ for any $t>0$. The extended L\'{e}vy measure function as defined in \cite{LiaoN} requires an additional assumption,
but on an irreducible $M=G/K$, it reduces to the above definition. The following result is \cite[Theorem~8.16]{LiaoN}, which may be regarded as an extension of Hunt's generator formula (\ref{hu2}) for L\'{e}vy processes
on an irreducible $G/K$ to inhomogeneous L\'{e}vy processes.
We will assume the basis $\{Y_1,\ldots,Y_n)$ of $\p$ is chosen as described above in relation to (\ref{hu2}).

\begin{theorem} \label{thmartrep}
Let $M=G/K$ be irreducible with $\dim(M)>1$. For any inhomogeneous L\'{e}vy process $X_t$ in $M$, there is a unique pair $(a,\nu)$ comprising a continuous nondecreasing function $a(t)$ with $a(0)=0$ and
an extended L\'{e}vy measure function $\nu(t,\cdot)$ such that for any $f\in C_c^\infty(M)$,
\begin{equation} \label{mart}
f(X_t) - \int_0^t\frac{1}{2}\sum_{i=1}^nY_iY_if(X_s)da(s) - \int_0^t\int_M[f(X_s\tau) - f(X_s)]\nu(ds\,d\tau)
\end{equation}
is a martingale under the natural filtration of $X_t$, where the integral $\int_M[\cdots]$ is understood as the principal value as in (\ref{hu2}). Moreover, $X_t$ is stochastically continuous if and only if $\nu$
is a L\'{e}vy measure function.

  Conversely, given any pair $(a,\nu)$ as above, there is an inhomogeneous L\'{e}vy process $X_t$ in $M$ with $X_0=o$, unique in distribution, that has the above martingale property. Moreover, if $\nu$
is a L\'{e}vy measure function, then the uniqueness holds among all \cadlag\ processes in $M$.
\end{theorem}

\section{Decomposition of Invariant Markov Processes}

Let $M$ be a manifold under the non-transitive action of a Lie group $G$. Then $M$ is a collection of $G$-orbits. Because $G$ acts transitively
on each $G$-orbit, any $G$-orbit is a homogeneous space $G/K$ for some closed subgroup $K$.  By the Principal Orbit Type Theorem (\cite{tomDieck}), there is a subgroup $K$ such that the $G$-orbits
of type $G/K$ fill up an open dense subset of $M$. Thus, by removing a small subset if necessary, we may assume $M$ is a union of $G$-orbits of type $G/K$. Then it is reasonable to
assume there is a submanifold $M_1$ of $M$ that is transversal to $G$-orbits in the sense that it intersects each $G$-orbit at exactly one point, and
\begin{equation} \label{MM1M2}
M = M_1\times M_2, \quad M_2 = G/K.
\end{equation}
The above will be assumed with $K$ being compact. Note that $G$ acts on $M$ through its natural action on $M_2=G/K$, and it acts trivially on $M_1$.

For example, the orthogonal group $O(n)$ acts non-transitively on $\R^n$. Its orbits are spheres centered at the origin of type $O(n)/O(n-1)$, except the origin is an orbit
by itself. After removing the origin, $\R^n$ becomes $M_1\times M_2$, where $M_1$ is a half line from the origin and $M_2$ is the unit sphere. This is just the usual spherical
polar decomposition.

Let $X_t$ be a $G$-invariant Markov in $M$, and let $X_t=(Y_t,Z_t)$, where $Y_t\in M_1$ and $Z_t\in M_2$, be the decomposition (\ref{MM1M2}). Borrowing the terms from the polar decomposition of $\R^2$,
we will call $Y_t$ the radial part and $Z_t$ the angular part of the process $X_t$, noting that both may be multi-dimensional. It is easy to show that $Y_t$ is a Markov process in $M_1$,
and with some effort, it can be shown that $Z_t$ is an inhomogeneous L\'{e}vy process in $M_2=G/K$ under the conditional distribution given $Y_t$, but to state this more precisely requires some preparation. See Sections 1.5 and 1.6
of \cite{LiaoN} for more details.

Let $D(M)$ be the space of \cadlag\ paths in $M$ equipped with the $\sigma$-algebra ${\mathcal F}$ generated by the coordinate maps $\omega\mapsto\omega(t)$, $t\geq 0$, for $\omega\in D(M)$. The distribution
of the process $X_t$ with $X_0=x$ induces a probability measure $P_x$ on $D(M)$. We may regard $X_t$ as the coordinate process on $D(M)$ by setting $X_t(\omega)=\omega(t)$. Similarly, $Y_t$ and $Z_t$
are regarded as coordinate processes on $D(M_1)$ and $D(M_2)$ (respectively). The product $M=M_1\times M_2$ induces a product $D(M)=D(M_1)\times D(M_2)$. Let $J_1$: $D(M)\to D(M_1)$ and $J_2$: $D(M)\to D(M_2)$
be the natural projections, and let ${\mathcal F}^1$ and ${\mathcal F}^2$ be respectively the $\sigma$-algebras on $D(M_1)$ and $D(M_2)$ generated by the coordinate maps.
We may regard ${\mathcal F}^1$ as a $\sigma$-algebra on $D(M)$ by identifying it with $J_1^{-1}({\mathcal F}^1)=\{J_1^{-1}(B)$; $B\in{\mathcal F}^1\}$.

Let $x=(y,z)\in M=M_1\times M_2$. By the $G$-invariance of the Markov process $X_t$, $Q_y=P_x\circ J_1^{-1}$ is a probability measure on $D(M_1)$ depending only on $y$, and is in fact the distribution of the radial process $Y_t$.

Because $(D(M),{\mathcal F})$ is a standard Borel space, there is a regular conditional distribution $P_x^{y(\cdot)}(\cdot)$ of $P_x$ given ${\mathcal F}^1$, that is, for any $F\in{\mathcal F}$,
\[P_x(F\mid{\mathcal F}^1)(y(\cdot)) = P_x^{y(\cdot)}(F) \quad \mbox{for $Q_y$-almost all $y(\cdot)\in D(M_1)$},\]
where $P_x^{y(\cdot)}(\cdot)$ is a probability kernel from $D(M_1)$ to $D(M)$ in the sense that it is a probability measure on $D(M)$ for each fixed $y(\cdot)\in D(M_1)$ and $P_x^{y(\cdot)}(F)$ is ${\mathcal F}^1$-measurable
in $y(\cdot)$ for each $F\in{\mathcal F}$.

Let $R_z^{y(\cdot)}=P_x^{y(\cdot)}\circ J_2^{-1}$. Then $R_z^{y(\cdot)}$ is a probability kernel from $D(M_1)$ to $D(M_2)$, and for $F_1\in{\mathcal F}^1$ and $F_2\in{\mathcal F}^2$,
\[P_x(F_1\cap F_2) = \int_{F_1}R_z^{y(\cdot)}(F_2)Q_y(dy(\cdot)).\]
Therefore, $R_z^{y(\cdot)}$ is the conditional distribution of the angular process $Z_t$ given the radial process $Y_t$. The following result is Theorems 1.31 and 1.39 in \cite{LiaoN}.

\begin{theorem} \label{thdecomp}
For $y\in M_1$ and $z\in M_2$, the radial process $Y_t$ is a Markov process in $M_1$ under $Q_y$, and for $Q_y$-almost all $y(\cdot)\in D(M_1)$, the angular process $Z_t$ is an inhomogenous L\'{e}vy process
in $M_2=G/K$ under $R_z^{y(\cdot)}$.
\end{theorem}

We will now consider a continuous $G$-invariant Markov process $X_t$ in $M=M_1\times M_2$ and assume $M_2=G/K$ is irreducible. An example is a Brownian motion in $\R^n$, which is invariant under the orthogonal group $O(n)$, and in that case, the associated decomposition $X_t=(Y_t,Z_t)$ is just the usual spherical polar decomposition.

It is well known that the radial process $Y_t$ and the angular process $Z_t$ are not independent, but they are independent up to a random time change. More precisely, there is a spherical Brownian motion $W_t$ in the unit sphere $S^{n-1}$,
independent of $Y_t$, and a time change process $a(t)$ adapted to the natural filtration of $Y_t$, such that $Z_t=w_{a(t)}$. This is called the skew-product decomposition of the Brownian motion.
Here the time change process $a(t)$ is a real-valued nondecreasing process with $a(0)=0$.

This skew-product is generalized by Galmarino \cite{Galmarino} to any continuous $O(n)$-invariant Markov process in $\R^n$. Using the decomposition
of invariant Markov processes (Theorem~\ref{thdecomp}) and the representation of inhomogeneous L\'{e}vy processes in irreducible homogeneous spaces (Theorem~\ref{thmartrep}),
we may obtain the skew-product on a more general space by a conceptually more transparent proof.  See \cite[Section~9.2]{LiaoN} for more details. Note that the time change process $a(t)$
in the following theorem is the same $a(t)$ as in Theorem~\ref{thmartrep}.

\begin{theorem} \label{thskewprod}
Let $X_t$ be a continuous $G$-invariant Markov process in $M=M_1\times M_2$ and assume $M_2=G/K$ is irreducible with $\dim(M_2)>1$. Then there are a Riemannian Brownian motion $W_t$
in $M_2$ under a $G$-invariant Riemannian metric, independent of the radial process $Y_t$, and a time change process $a(t)$ that is adapted to the natural filtration of $Y_t$, such that $Z_t=W_{a(t)}$.
\end{theorem}

\end{document}